\documentclass[12pt]{article}
\def\eps{{\varepsilon}}
\usepackage{epsfig}
\usepackage{amsmath,amsthm,amsfonts,amssymb}
\usepackage{subcaption}
\usepackage{graphicx}
\usepackage{float}
\usepackage{enumitem} 

\allowdisplaybreaks

\newtheorem{theorem}{Theorem}[section]
\newtheorem{lemma}[theorem]{Lemma}
\newtheorem{corollary}[theorem]{Corollary}
\newtheorem{remark}[theorem]{Remark}
\newtheorem{proposition}[theorem]{Proposition}
\newtheorem{definition}[theorem]{Definition}

\begin{document}

\title{On the graphs having at most one  positive eccentricity eigenvalue}

\author{
Sezer Sorgun\thanks{\small Dept. of Mathematics,
Nev\c{s}ehir Hac{\i} Bekta\c{s} Veli University, Turkey}
\\
\small{\tt{srgnrzs@gmail.com}}
\\[5pt]
Hakan Küçük
\footnotemark[1]\\
\small{\tt{hakankucuk1979@gmail.com}}
\\[5pt]
}
\date{}
\maketitle
\begin{abstract}
\noindent
The eccentricity (anti-adjacency) matrix $\eps(G)$ of a graph $G$ is obtained from the distance matrix by retaining the eccentricities in each row and each column. This matrix is first defined in 2018 by Wang et al. \cite{1}.

In this paper we have characterized the graphs which have at most one (hence exactly) positive eigenvalue of $\eps(G)$.
\bigskip


\end{abstract}

\section{Introduction}
In this paper we considered graphs as simple, connected graphs. Let $G=(V(G),E(G))$ be a graph with vertex set $V(G)=\{v_{1},v_{2},\ldots,v_{n}\}$ and edge set $E(G)=\{e_{1},e_{2},\ldots,e_{n}\}$. if two vertices $v_{i}$ and $v_{j}$ are adjacent, then we denoted by $v_{i}\sim v_{j}$ and the edge between two vertices is denoted by $e_{ij}$.
A path in a graph is the any route along the edges of the graph. The $distance$ between the vertices $u,v \in V(G)$, denoted by $d_{G}(u,v)$, is the minimum length of the paths between $u$ and $v$. The \textit{diameter} $diam(G)$ of G is the maximum eccentricity among its vertices and the \textit{radius} $rad(G)$ is the minimum eccentricity of its vertices. Any $u,v \in V(G)$ is \textit{diametral pair} if $d_{G}(u, v) = diam(G)$. A diametral path of a graph is a shortest path whose length is equal to the diameter of the graph.  The \textit{eccentricity} $e(u)$ of the vertex $u$ is defined as $e(u)=max\{d_{G}(u,v):v \in V(G)\}$.  If $diam(G)=rad(G)=d$ than $G$ is known $d$-self centered graph.

The $eccentricity$ $matrix$ $\varepsilon(G)=(\epsilon_{uv})$ of a graph $G$ is defined as follows:\\

\begin{equation}
  \epsilon_{uv}=\begin{cases}
    d_{G}(u,v), & \text{if $d_{G}(u,v)=min\{e(u),e(v)\}$}.\\
    0, & \text{otherwise}.
  \end{cases}
\end{equation}
It is also known as anti adjacency matrix \cite{1}. Since $\eps(G)$ is a symmetric matrix, then eigenvalues of $\varepsilon(G)$ are real. Let $\xi_{k}>\xi_{k-1}>\ldots>\xi_{1}$ be the  distinct $\varepsilon-eigenvalues$. Then the $\varepsilon-spectrum$ can be given as\\
\begin{equation*}
   spec_{\varepsilon}(G)=\big\lbrace\xi_{1}^{(t_1)},\xi_{2}^{(t_2)},\ldots,\xi_{k}^{(t_k)}
   \big\rbrace
\end{equation*}
where $t_{i}$ is the multiplicity of each eigenvalue $\xi_{i} (1\leq i \leq n)$. When $G$ is a connected graph of order $n$, diameter $d=2$ and $\Delta< n-1$,  since $\epsilon_{uv}=2$ if $uv\in E(G)$ and $\epsilon_{uv}=0$ if $uv\notin E(G)$ . Hence,
$$\eps(G) = 2A(\overline{G})$$
where $\overline{G}$ is the complement of $G$ \cite{1}\\
The pineapple graph $K_{p}^{q}$ is the coalescence of the complete graph $K_{p}$ (at any vertex) with the star $K_{1,q}$, at the vertex of degree q.\\
The Kite graph, denoted by $Kite_{p,q}$, is obtained by appending a complete graph with p vertices $K_{p}$ to a pendant vertex of a path graph with q vertices $P_{q}$. For $q=1$, then graph is known as short kite graph.\\
A complete split graph $CS(n,\alpha)$ is a graph on $n$ vertices consisting of a clique on $n-\alpha$ vertices and an independent set on the remaining $\alpha$  vertices $(1 \leq \alpha \leq n-1)$ in which each vertex of the clique is adjacent to each vertex of the independent set.\\

The Windmill graph $W_{m}^{l}$ is the graph obtained by taking  $l$-copies of the complete graph $K_{m}$ with a vertex in common. For $m=3$, the graph is known as Friendship graph $F_{3}^{l}$ (also, Dutch windmill).\\

In literature, there are many graph matrices (are known M-matrix) such as adjacency, Laplacian, signless Laplacian etc. But the eccentricity matrix of a graph is  a quite interesting graph matrix apart from them. The most interesting aspect of the eccentricity matrix of the graphs is that interlacing lemma which is applicable for the other M-matrices and well known in the world of mathematics, cannot be applied. This lemma is very useful for spectral properties of graphs and is almost one of the most important elements. For the graph matrices, the induced subgraph of a graph corresponds to the principal submatrix of the matrix, while for the eccentricity matrix, this feature may not always be present. For instance,
in Figure 1, $H$ is an induced subgraph of $G$. 

\begin{figure}[H]
	\centering
	\begin{subfigure}{.3\textwidth}
		\includegraphics[width=\textwidth]{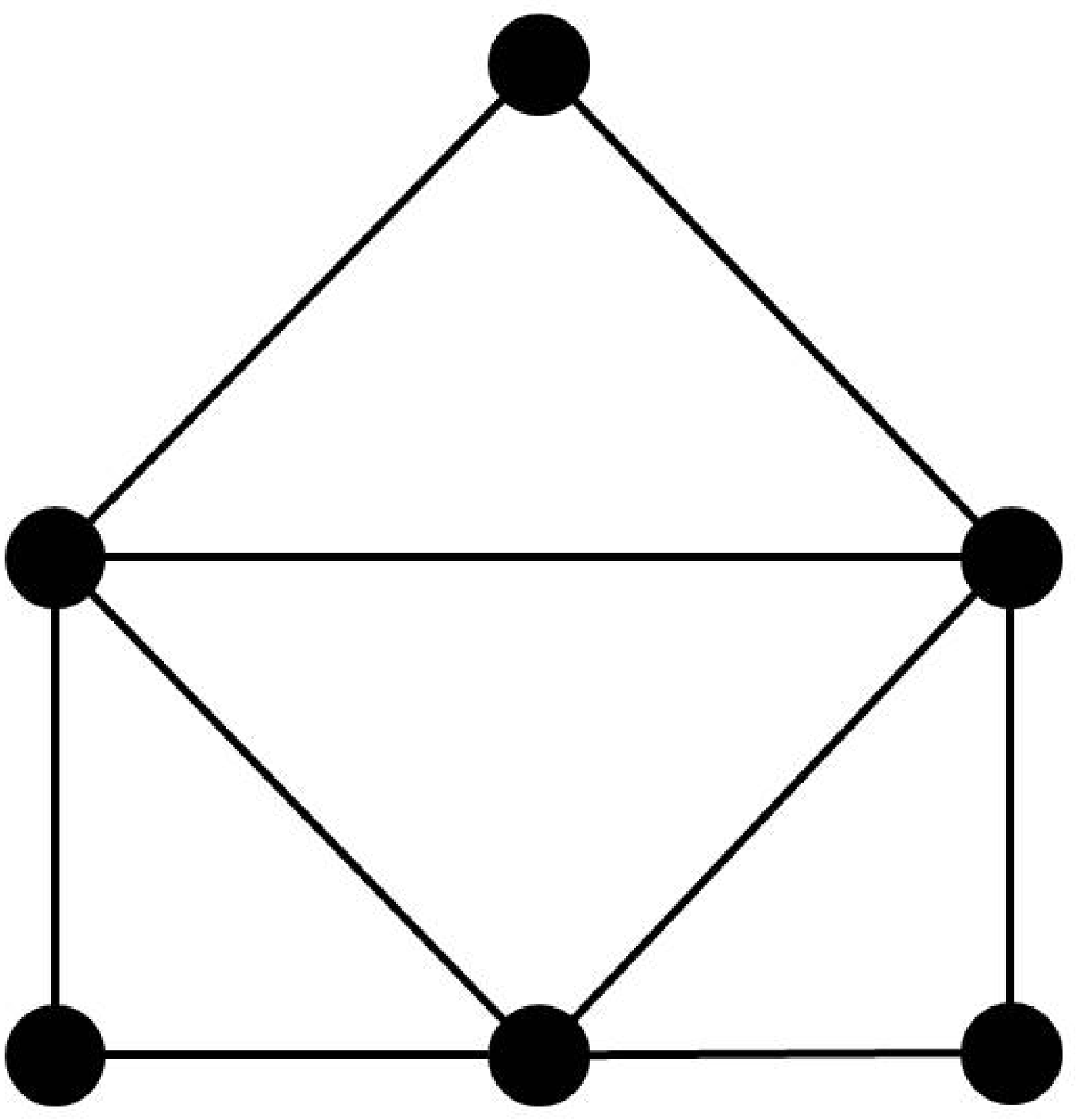}
		\caption{$G$}
        \hfill
	\end{subfigure}
	\begin{subfigure}{.3\textwidth}
		\includegraphics[width=\textwidth]{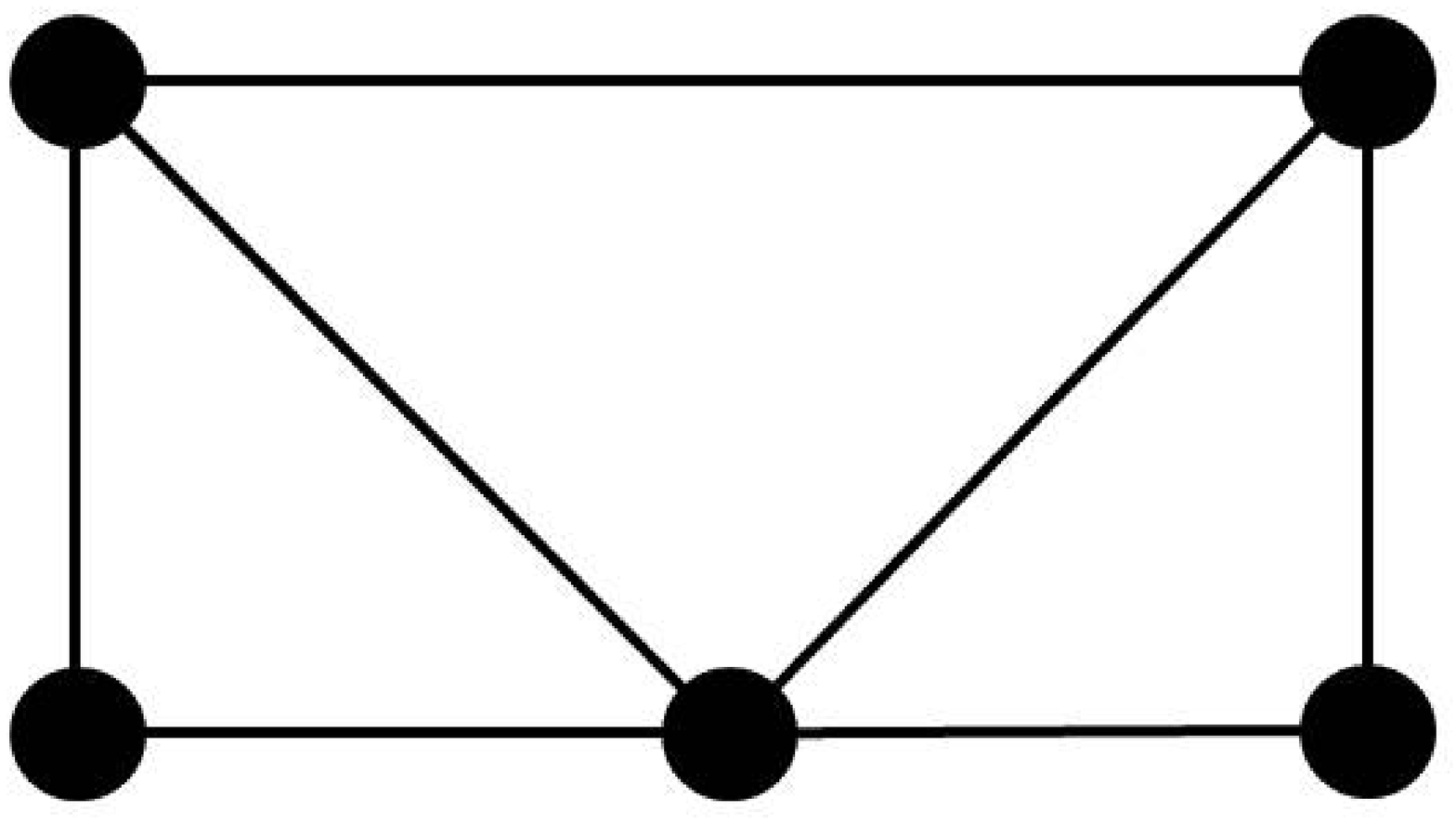}
		\caption{$H$}
        \hfill
	\end{subfigure}
	\caption{}
\end{figure}

But $\eps(H)$ is not a principal submatrix of $\eps(G)$ as seen below.
\[
\eps(G)=\left[
\begin{array}{cccccc}
0 &0&2&2&0&2\\
0 &0&0&2&0&0\\
2 &0&0&0&0&0\\
2 &2&0&0&0&2\\
0 &0&0&0&0&2\\
2 &0&0&2&2&0
\end{array}
\right],\eps(H)=\left[
\begin{array}{ccccc}
0 &0&2&2&1\\
0 &0&0&2&1\\
2 &0&0&0&1\\
2 &2&0&0&1\\
1 &1&1&1&0
\end{array}
\right],
\]\\

Because of the definition of the eccentricity matrix, the entries of it is not deterministic, hence the spectral characterization problem is quite complicated. The matrix is a new matrix and there are a few paper regarding spectral properties. In \cite{8} Henning et.al. give a characterization  graph special graph families with diameter $2$. Wang et. al \cite{1} give some result of $r$-regular graph with diameter $2$ w.r.t. eccentricity matrix. There are some papers regarding $\eps$- spectra of some standart graphs such as star $K_{1,k-1}=S_{k}$, complete multipartite graph $K_{n_1,n_2,\ldots n_k}$,  complete graph $K_n$ etc.(For detail, see \cite{1,3,4}. Also in \cite{10} it is shown that $K_n$  and $K_{n_1,n_2,\ldots,n_k}$ are determined by $\eps$-spectrum. Among all graphs with diameter $d=2$,  $S_{n}$ is one of the graphs which have exactly one positive $\eps$-eigenvalue.

By the motivation of the graph $\eps$-spectra, we have characterized the graphs having exactly one positive $\eps$-eigenvalue.

\section{Main Results}

\begin{lemma} \cite{6}
Let $A\in M_n(\mathbb{R})$ be a symmetric matrix. Let  $B\in M_m(\mathbb{R})$ be a principal submatrix of $A$. Suppose that  $A$ and $B$ has eigenvalues $\lambda_1\leq \lambda_2\leq \cdots \leq\lambda_n$ and  $\beta_1\leq \beta_2\leq \cdots \leq\beta_m$, respectively. Then
$$\lambda_k\leq \beta_k\leq \lambda_{k+n-m}$$ for $k=1,\ldots m$.
\end{lemma}
\begin{definition}\cite{9}
Let $A$ be partitioned according to ${X_1,\ldots,X_m}$, that is,
\[
A=\left[
\begin{array}{ccc}
A_{1,1} & \cdots &A_{1,n}\\
\vdots & \vdots& \vdots\\
A_{m,1} & \cdots & A_{m,m}
\end{array}
\right],
\]
wherein $A_{i,j}$ denotes the submatrix (block) of $A$ formed by rows in $X_i$ and the columns in $X_j$. Let bi,j denote the average row sum of $A_{i,j}$. Then the matrix $B= (b_{i,j})$ is called the quotient matrix of $A$ w.r.t. the given partition.\\
When the row sum of each block $A_{i,j}$ is constant then the partition is called equitable.
\end{definition}

\begin{lemma}\cite{9}
Let $Q$ be a quotient matrix of any square matrix $A $ corresponding to an equitable partition. Then the spectrum of $A$ contains the spectrum of $Q$ .
\end{lemma}

\begin{theorem}\cite{zoran}
If $G$ is $d$–self centered, ($d\geq 2$) then any maximal circuit in $G$ consists of at least $2d$ vertices. Also, in the case $d=2$, a maximal circuit has length $4$ if and only if
$G$ is a complete bipartite graph with $2$ vertices in one partition.
\end{theorem}

\begin{lemma}
There is no $d$-self centered graph having at most one (hence exactly) positive $\eps$-eigenvalue.
\end{lemma}

\begin{proof}
Let $G$ be a self centered graph with diameter $d$ and let $V=\{v_1,v_2,\ldots, v_n\}$ be the vertex set. From Theorem 2.4, $G$ consists of at least $2d$ circuits. This also guarantees that $G$ has at least two diametral vertex pair such that $\{v_1,v_2\}$ and $\{v_3,v_4\}$. Hence, the eccentricity matrix of $G$  can be blocked as 
	
	\begin{eqnarray*}
	\eps(G)=\left[
	\begin{array}{c|c}
	A_{11} & .\\
	\hline
	.& .
	\end{array}
	\right]
	\end{eqnarray*}
	where 
	\begin{eqnarray}
	\small 
	A_{11}=\left[
	\begin{array}{cccc}
	\small 0& d & 0 & 0\\
	\small d& 0 & 0 & 0\\
	\small 0& 0 & 0 & d\\
	\small 0& 0 & d & 0 
	\end{array}
	\right],
	\end{eqnarray}
From the spectrum of the matrix in (2) and interlacing lemma, we get $\eps(G)$ has at least two positive $\eps$-eigenvalues.

\end{proof}

\begin{lemma}
There is no graph (not self centered) having at most one (hence exactly) positive $\eps$-eigenvalue for $diam(G)\geq 3$.
\end{lemma}

\begin{proof}
	
 Let $G$ be a not d-self centered graph let $V=\{v_1,v_2,\ldots, v_n\}$ be the vertex set. Then, there is at least one diametral path. Let's label a diametral path $P=v_1v_3\ldots v_4v_2$.  Then, the matrix $\eps(G)$  can be blocked as the following,
	\[
	\eps(G)=\left[
	\begin{array}{c|c}
     A'_{11} & .\\
     \hline
     .& .
	\end{array}
	\right]
	\]
where 	\begin{eqnarray}
\small 
A'_{11}=\left[
\begin{array}{cccc}
\small 0& d & 0 & d-1\\
\small d& 0 & d-1 & 0\\
\small 0& d-1 & 0 & 0\\
\small d-1& 0 & 0 & 0 
\end{array}
\right]
\end{eqnarray}
Since the matrix in  (3) has rank 4 and moreover it can also be  blocked by $2$ by $2$ , they have symmetric non-zero eigenvalues. By Lemma 2.1, we get the matrix $\eps(G)$  has at least two positive $\eps$-eigenvalues. 
\end{proof}

\begin{figure}[ht]
\begin{subfigure}{.5\textwidth}
  \centering
  \includegraphics[width=.8\linewidth]{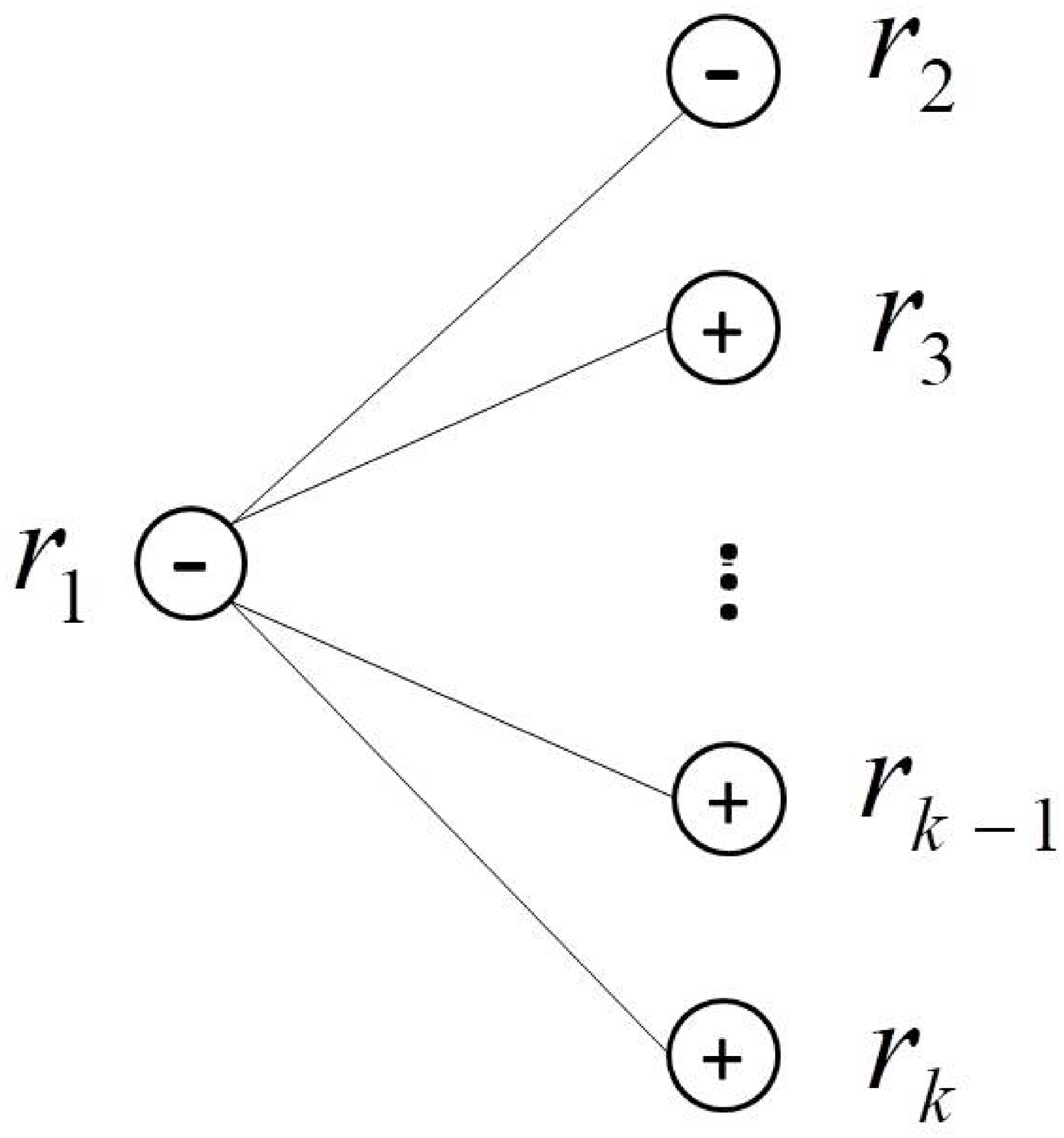}
  \caption{$(-r_1,-r_2,r_3,\ldots,r_k)$}
  \label{fig:sub-first}
\end{subfigure}
\begin{subfigure}{.5\textwidth}
  \centering
  \includegraphics[width=.8\linewidth]{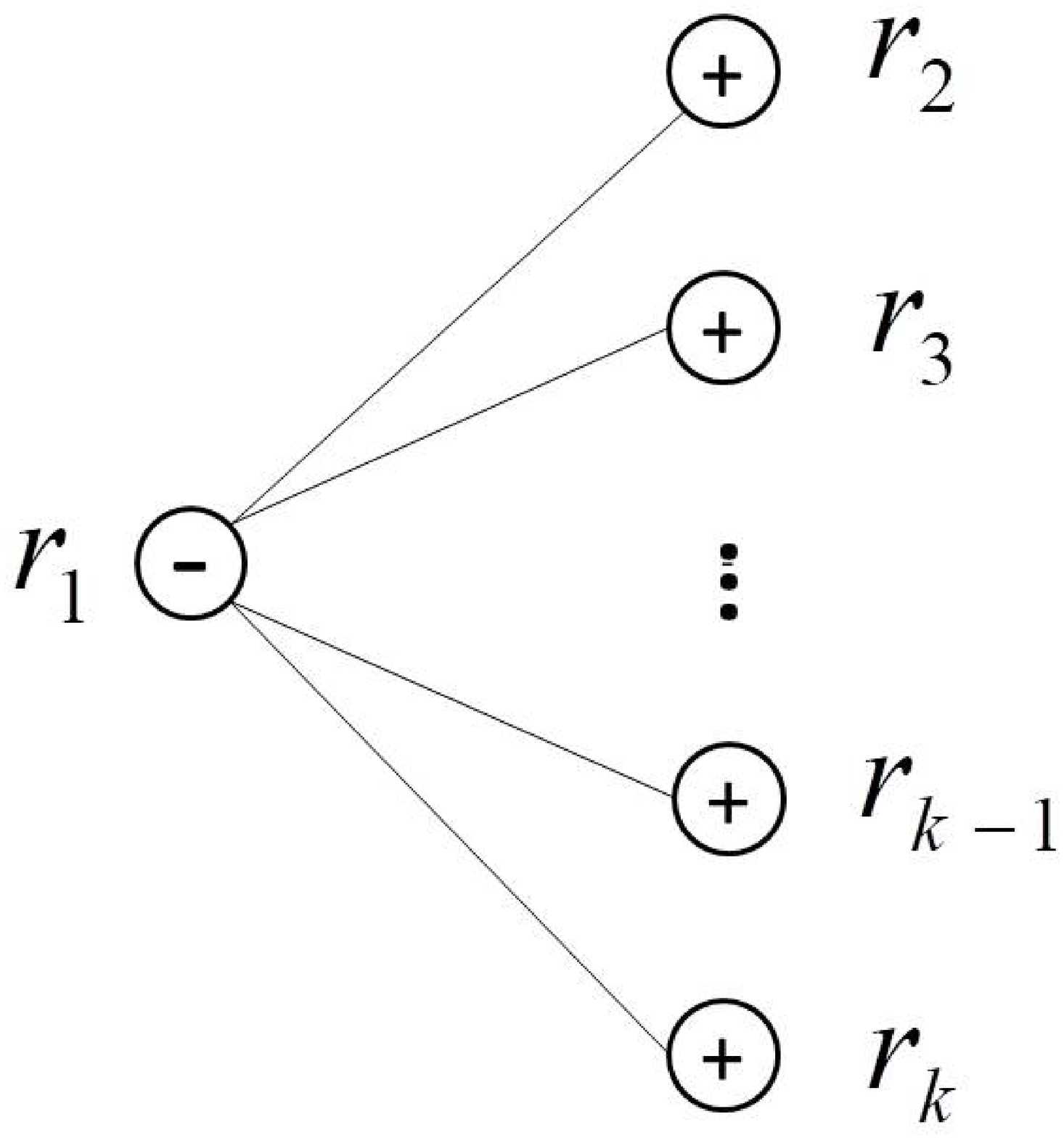}
  \caption{$(-r_1,r_2,r_3,\ldots,r_k)$}
  \label{fig:sub-second}
\end{subfigure}

\begin{subfigure}{.5\textwidth}
  \centering
  \includegraphics[width=.8\linewidth]{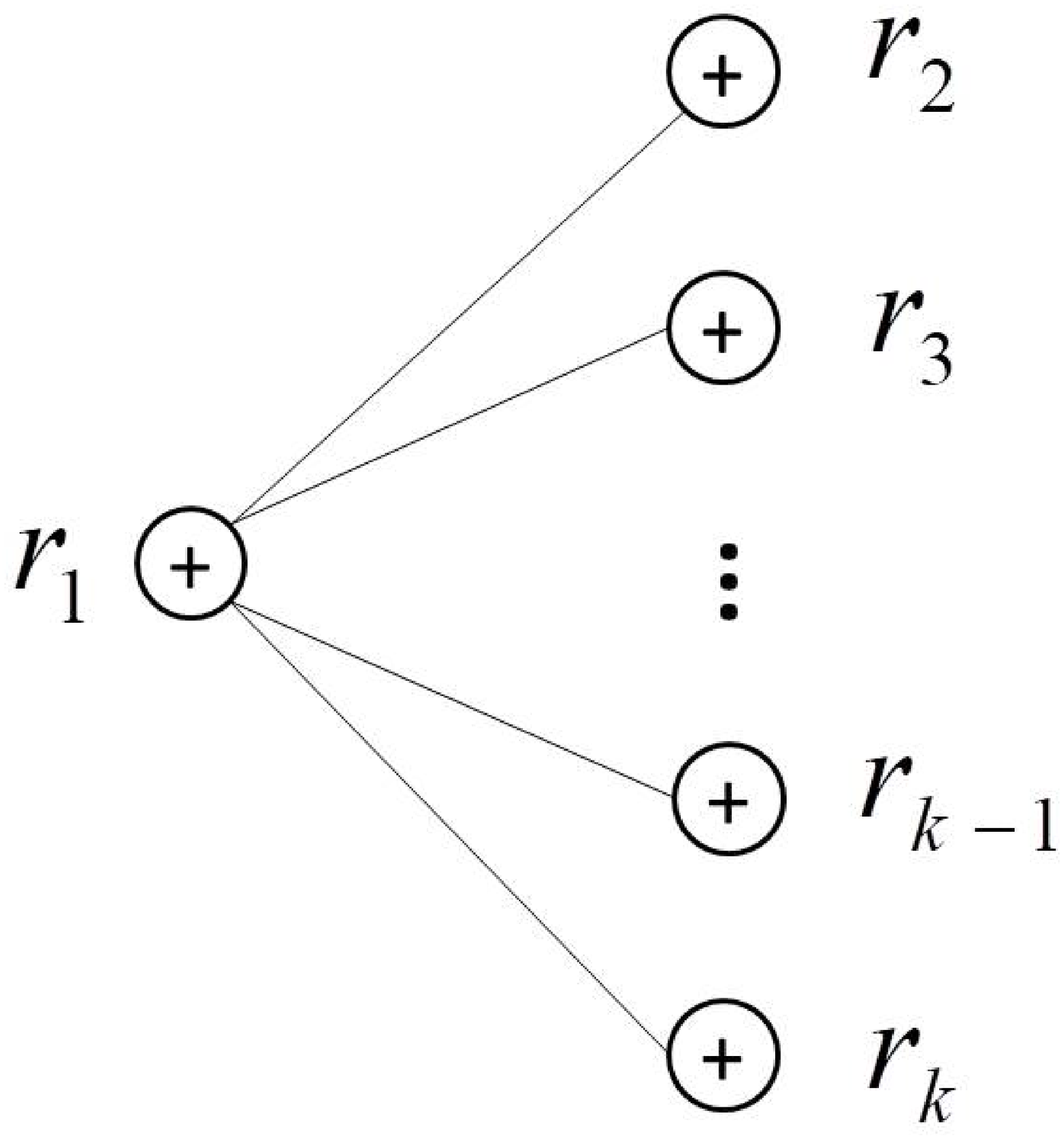}
  \caption{$(r_1,r_2,r_3,\ldots,r_k)$}
  \label{fig:sub-third}
\end{subfigure}
\begin{subfigure}{.5\textwidth}
  \centering
  \includegraphics[width=.8\linewidth]{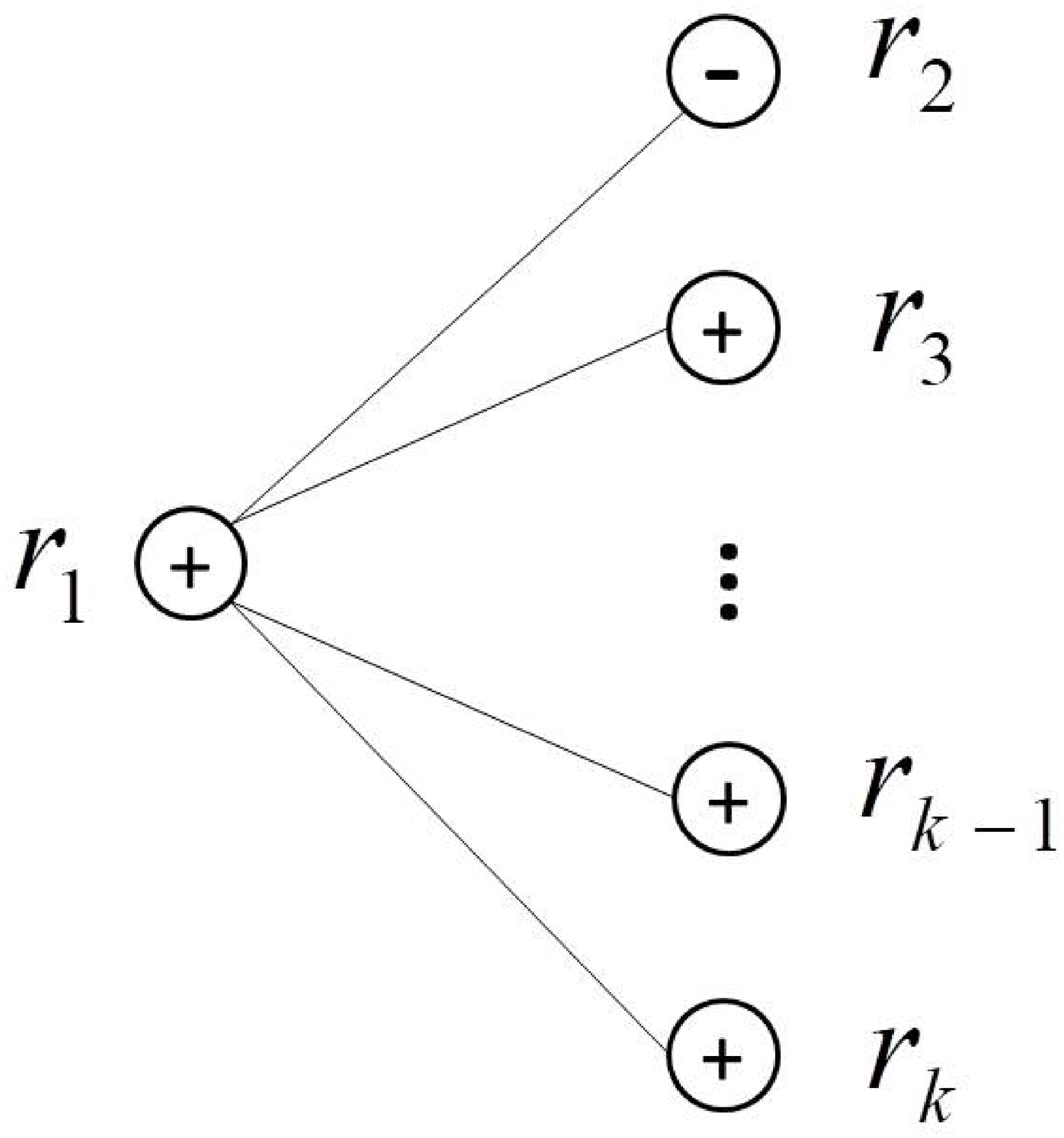}
  \caption{$(r_1,-r_2,r_3,\ldots,r_k)$}
  \label{fig:sub-fourth}
\end{subfigure}
\caption{All possible mixed extension of star $S_{k}$}
\label{fig:fig}
\end{figure}

Let $G$ be any graph with diameter d or $\infty$ and $\mathcal{G}=\{G+u: G \hspace{2mm}\textit{\textit{any graph}} \}$ be family. It is easy to see that all graphs with diameter 2 and $\Delta=n-1$ are in $\mathcal{G}$.
\begin{theorem}
Let $G$ be a graph of order $n$ and $V=\{v_1,v_2,\ldots,v_n\}$. If $diam(G)\geq 3$ or $rad(G)=diam(G)=2$, then $G+v_{n+1}$ has at least two positive $\eps$-eigenvalue.
\end{theorem}
\begin{proof}
	Assume that $diam(G)=rad(G)=2$. Then the eccentricity matrix of $G+v_{n+1}$
	\begin{eqnarray}
	\eps(G+v_{n+1})=\left[
	\begin{array}{cc}
	\eps(G) &J \\
	J &0
	\end{array}
	\right],
	\end{eqnarray}\\
	where $J$ is all 1 matrix by 1 by n. From Lemma 2.5 and Lemma 2.1 we get $G+v_{n+1}$ has at least two positive $\eps$-eigenvalue.
	\vspace{2mm}
	Now let $diam(G) \geq 3$. Let's label the diametral path as  $v_1v_3\ldots v_4v_2$. So we get
	
	\[
	\eps(G+v_{n+1})=\left[
	\begin{array}{c|c}
	U & .\\
	\hline
	.& .
	\end{array}
	\right]
	\]\\

where \small \[
U=\left[
\begin{array}{cccc}
\small 0& 2 & 0 & 2\\
\small 2& 0 & 2 & 0\\
\small 0& 2 & 0 & 0\\
\small 2& 0 & 0 & 0 
\end{array} 
\right] \textit{or} 	 \ U=\left[
\begin{array}{cccc}
\small 0& 2 & 0 & 2\\
\small 2& 0 & 2 & 0\\
\small 0& 2 & 0 & 2\\
\small 2& 0 & 2 & 0 
\end{array}
\right]\]

Since the matrix $U$ has rank 4 and also they can be  blocked by 2 by 2. It has symmetric non-zero eigenvalues. This fact guarantees that the matrix $U$ has exactly two positive eigenvalues. From Lemma 2.1, we get the matrix $\eps(G)$  has at least two positive $\eps$-eigenvalues. 
\end{proof}

A mixed extension of a graph $G$ is a graph $H$ obtained from $G$ by replacing each vertex of $G$ by a clique or a coclique, whilst two vertices in $H$ corresponding to distinct vertices $x$ and $y$ of $G$ are adjacent whenever $x$ and $y$ are adjacent in $G$ \cite{5,7,11}. A mixed extension of star $S_k$ is represented by an $k$-tuple $(\pm r_1,\ldots \pm r_k)$ of nonzero integers, where $+r_i $ (simply $r_i$) indicates that $V_{i}$ is a clique of order $r_i$ and $-r_i $ means that $V_{i}$ is a coclique of order $r_i$. Without loss of generality, let  $r_1$ be representation of the center vertex of $S_{k}$. Therefore the possible types of mixed extension can be seen in Fig 1.\\ 

\begin{proposition}
Let $H$ be a mixed extension of $S_{2}$ of any type . Then $H$ has an exactly one positive $\eps$-eigenvalue if and only if the type must be one of the following:

\begin{enumerate}[label=(\roman*)]
\item $(-2,r_{2})$  for all $r_2\geq 1$
\item $(-r_1,2)$  for all $r_1\geq 2$
\item $(-3,4)$ or $(-4,3)$ or $(-3,3)$
\item $(r_1,r_{2})$ for all $r_1,r_2\geq1$
\end{enumerate}

\end{proposition}

\begin{proof}

Let $H=(-r_1,-r_2)$. Then $H\cong K_{r_1,r_2}$. From the spectrum of $K_{r_1,r_2}$, $H$ has two positive $\eps$-eigenvalue. Hence $H$ can be type of $(r_1,r_2)$;$(-r_1,r_2)$ \\

Assume that $H=(r_1,r_2)$. Then it is easy to see that $H\cong K_{r_1+r_2}$. Hence $H$ has exactly one positive $\eps$-eigenvalue because $H$ has rank 1.\\

Now let $H=(-r_{1},r_{2})$. The eccentricity block matrix of $H$ is\\
\[
\eps(H)=\left[
\begin{array}{cc}
2(J-I)_{r_1\times r_1} &J_{r_{1}\times r_{2}}\\
J &  (J-I)
\end{array}
\right],
\]\\
Since the given partitions of the matrix $H$ are equitable, hence from the quotient matrix, we get the characteristic polynomial of the mixed extension of types as the following:

\begin{equation*}
p_{\eps}(H,x)=(x+2)^{r_{1}-1}(x+1)^{r_{2}-1}h(x)
\end{equation*}
where $h(x)=x^2-(2r_1+r_2-3)x+r_1r_2-2r_1-2r_2+2$. From the product of the roots (say $\alpha_1$ and $\alpha_2$) of polynomial $h(x)$, we have $\alpha_1\alpha_2=r_1r_2-2r_1-2r_2+2\leq 0$ if and only if $r_1r_2+2\leq 2r_1+2r_2$ if and only if $(-r_1,r_2)\in \{(-2,r_2),(-3,4),(-4,3),(-3,3),(-r_1,2)\}$.
\end{proof}

\begin{proposition}
Let $H$ be a mixed extension of $S_{3}$ of any type . Then $H$ has an exactly one positive $\eps$-eigenvalue if and only if the type must be one of the following:

\begin{enumerate}[label=(\roman*)]
\item $(r_1,r_2,r_3)$  for all $r_i\geq 1$
\item $(r_1,-r_2,r_3)$  for all $r_2\geq 2$,$r_1\leq2$ or $(r_1,r_2)\in\{(3,2),(4,2),(3,3)\}$
\end{enumerate}

\end{proposition}

\begin{proof}

 The all mixed extensions of $S_3$ is type of  $(r_1,r_2,r_3)$ ,$(r_1,-r_2,r_3)$, $(-r_1,r_2,r_3)$ and $(-r_1,-r_2,r_3)$. For the type of $(-r_1,r_2,r_3)$ and $(-r_1,-r_2,r_3)$ , $r_1$ cannot be greater than or equal 2 because of Lemma 2.5. Hence the only types must be $(r_1,r_2,r_3)$ and $(r_1,-r_2,r_3)$ for $r_1\geq1$. 
 
Let $H_1$ and $H_2$ be the mixed extensions of $S_3$ type of  $(r_1,r_2,r_3)$  and $(r_1,-r_2,r_3)$ respectively.\\
The eccentricity block matrices of these types  are
\[
\eps(H_1)=\left[
\begin{array}{ccc}
(J-I)_{r_1} & J &J\\
J & 0_{r_2} & 2J\\
J & 2J & 0_{r_3}
\end{array}
\right],\eps(H_2)=\left[
\begin{array}{ccc}
(J-I)_{r_1} & J &J\\
J & 2(J-I)_{r_2} & 2J\\
J & 2J & 0_{r_3}
\end{array}
\right]
\], where $J$ is all one  matrix and $I$ is the identity matrix.

Since the given partitions of the matrices are equitable, hence from the quotient matrices we get the characteristic polynomial of the mixed extension of the types  as the following:
\begin{equation}
p_{\eps}(H_1,x)=x^{r_2+r_3-2}(x+1)^{r_1-1}h_1(x)
\end{equation}
where $h_1(x)=x^3-(r_1-1)x^2-(r_1r_2+r_1r_3+4r_2r_3)x-4r_2r_3$.

\begin{equation}
p_{\eps}(H_2,x)=x^{r_3-1}(x+1)^{r_2-1}(x+2)^{r_1-1}h_{2}(x)
\end{equation}
where $h_{2}(x)=x^3-(2r_2+r_1-3)x^2+(r_1r_2-4r_2r_3-r_1r_3-2r_2-2r_1+2)x+2r_3(r_1r_2-2r_2-r_1)$. \\

Let $\alpha_1$,$\alpha_2$,$\alpha_3$ be  the roots of polynomials of $h_{i}$. Notice that the remaining roots of $p_{\eps}(H_i,x)$ are either 0 or negative real numbers. Moreover, since these roots are also $\eps$-eigenvalues with respect to graphs $H_i$ and  $\sum_{i=1}^{n}\xi_{i}=0$, then at least one of these roots must be positive.\\
\vspace{2mm}
Assume that $\alpha_1>0$. Now let's start the proof for each case:
\item[i.] By the product of the roots of polynomial $h_1(x)$ we get
\begin{eqnarray*}
\alpha_1\alpha_2\alpha_3 & = & 4r_2r_3 \nonumber \\
& > &0 \nonumber\\
& \Leftrightarrow & \alpha_2>0,\alpha_3>0 \hspace{2mm}\text{or} \hspace{2mm} \alpha_2<0,\alpha_3<0  \nonumber
\end{eqnarray*}
Suppose that $\alpha_2>0,\alpha_3>0$. Then we have $\alpha_1\alpha_2+\alpha_1\alpha_3+\alpha_2\alpha_3>0$. But this is a contradiction because the sum of the product of the roots of $h_1(x)$ is negative. Hence $\alpha_2<0,\alpha_3<0$. So $H_1$ has exactly one positive $\eps$-eigenvalue which is $\alpha_1$.\\
\item[ii.] Similarly, we have
\begin{eqnarray*}
  \alpha_1\alpha_2\alpha_3 & = & 2r_3(2r_2+r_1-r_1r_2) \nonumber
\end{eqnarray*}
Then $ 2r_3(2r_2+r_1-r_1r_2)>0$ if and only if $2r_2+r_1>r_1r_2$ if and only if
\begin{equation}
r_2\geq2 \hspace{2mm} \textsf{and} \hspace{2mm}  r_1\leq2;\hspace{2mm} (r_1,r_2)\in\{(3,2),(4,2),(3,3)\}
\end{equation}

Now assume that $\alpha_2\geq 0$ and $\alpha_3\geq 0$. Again by  the sum of the product of the roots of $h_2(x)$, we have
\begin{eqnarray}
  \alpha_1\alpha_2+\alpha_1\alpha_3+\alpha_2\alpha_3 & = & (r_2-4r_3-2)r_1-(r_3+2)r_2+2 \geq 0
\end{eqnarray}
On the other hand, under the conditions in (7) we have \begin{eqnarray*}
  (r_2-4r_3-2)r_1-(r_3+2)r_2+2 & \leq & (r_2-4r_3-2)r_1-(r_3+2)1+2 \nonumber \\
   & \leq & -(4r_1r_3+r_3) \nonumber \\
    & \leq & 0 \nonumber
\end{eqnarray*}
But this is a contradiction by (8). Hence we get $\alpha_2< 0$ and $\alpha_3<0$

\end{proof}

\begin{remark}
Notice  that we have  $(-r_1,r_2)\cong CS(r_1+r_2,r_1)$; $(-r_1,-1,r_3)\cong K_{r_3+1}^{r_1}$; $(1,r_2,1)\cong Kite_{r_2+1,1}$. Therefore, complete split graphs $CS(n,2)$, $CS(n+2,n)$, $CS(7,3)$, $CS(7,4)$; all pineapple graphs $K_{p}^{q}$ and short kite graph $K_{p,1}$ have exactly one positive $\eps$-eigenvalue.
\end{remark}

\begin{theorem}

Let  $\mathcal{H}$ denote the class of connected graphs  which formed in Fig 1. A graph $H \in \mathcal{H}$ has exactly one positive $\eps$-eigenvalue if and only if $H$ is one of the following.

\item{i.)}  A mixed extension of $S_{k}$ of  type $(1,-r_2,r_3,\ldots, r_k)$ for all $r_i\geq1$, $k\geq 4$.
\item[ii.)]  A mixed extension of $S_{k}$ of type $(1,r_2,r_3,\ldots,r_k)$ for all $r_i\geq1$, $k\geq 4$.

\item[iii.)]A mixed extension of $S_{k}$ of type $(2,r_2,r_3,\ldots, r_k)$ for all $r_i\geq1$,$k\geq 4$; $(4,r_2,r_3,r_4)$ for $r_2,r_3,r_4\geq1$; $(3,r_2,r_3,r_4,r_5)$ for $r_2,r_3,r_4,r_5\geq1$.
\item[iv.)] A mixed extension of $S_{k}$ of type $(2,-r_2,r_3,\ldots, r_k)$ for all $r_i\geq1$,$k\geq 4$ ; $(4,-r_2,r_3,r_4,r_5)$ for all $r_i\geq1$ ; $(3,-r_2,r_3,r_4,r_5,r_6)$ for all $r_i\geq1$.

\end{theorem}

\begin{proof}
\item[\textit{i.)}]
Assume that $r_1\geq 2$. Then $\Delta(H)<n-1$ (that's mean, $G$ is 2-self centered) , so $H$ has at least two positive $\eps$-eigenvalue by Lemma 2.5. Hence $r_1=1$. \\ For $r_1=1$, the eccentricity matrix of $H$ is

\begin{eqnarray}
\eps(H)=\left[
\begin{array}{ccc}
0_{1\times 1} &J_{1 \times r_2}&J_{1\times t}\\
J_{r_2\times 1} &2(J-I)_{r_2\times r_2}&2J_{r_2\times t}\\
J_{t\times 1} &2J_{t\times r_{2}}&X_{t\times t}
\end{array}
\right], ( t=r_3+r_4+\ldots+r_k)
\end{eqnarray}\\
where $X=\left[
\begin{array}{cccc}
0_{r_3\times r_3} &2J_{r_{3}\times r_{4}}&\ldots &2J_{r_{3}\times r_{k}}\\
2J_{r_4\times r_3} &0_{r_{4}\times r_{4}}&\ddots &2J_{r_{3}\times r_{k}}\\
\vdots &\ddots&\ddots &\vdots\\
2J_{r_k\times r_3} &2J_{r_{3}\times r_{4}}&\ldots &0_{r_{k}\times r_{k}}
\end{array}
\right]$. \\

The matrix $X$ admits an equitable partition into $k-2$ classes, and therefore $\eps(H)$ has an equitable partition into $k$ classes. Let $Q$ be the quotient matrix of the partition. Consider $Q'=Q+D$, where  $D=diag(1/2,2,2r_3,2r_4,\ldots,2r_k)$, then $Q'$ has rank 1, so there is only one nonzero eigenvalue, which is obviously positive. But the eigenvalues of $Q=Q'-D$ are smaller than those of $Q'$, which implies that $Q$ has at most (hence exactly) one positive eigenvalue. That's; $H\in \mathcal{H}$. \\

\item[\textit{ii.)}]  Let $H$ be the mixed extension of $S_{k}$ of type $(-r_1,r_2,r_3,\ldots, r_k)$. Similarly, when $r_1\geq 2$, $\Delta <n-1$, so from Lemma 2.5, we get that $H$ has at least two positive $\eps$-eigenvalue. Hence $r_1=1$. For $r_1=1$, the eccentricity matrix of $H$ is

\begin{eqnarray}
\eps(H)=\left[
\begin{array}{cc}
0_{1\times 1} &J_{1 \times t}\\
J_{t\times 1} &X_{t\times t}
\end{array}
\right], ( t=r_2+r_3+r_4+\ldots+r_k)
\end{eqnarray}\\

 Similarly the matrix $Q'=Q+D$, where $D=diag(1/2,2r_2,2r_3,\ldots,2r_k)$, then $Q'$ has rank 1, so there is only one positive eigenvalue. Hence $H\in \mathcal{H}$.\\

\item[\textit{iii.)}] Let $H$ be the mixed extension of $S_k$ type $(r_1,r_2,r_3,\ldots, r_k)$. Then  the eccentricity matrix of $H$ is
 \begin{eqnarray}
\eps(H)=\left[
\begin{array}{cc}
(J-I)_{r_1\times r_1} &J_{r_1 \times t}\\
J_{t\times r_1} &X_{t\times t}
\end{array}
\right], ( t=r_2+r_3+r_4+\ldots+r_k)
\end{eqnarray}\\

Also, the graph $H$ contains the complete split graph $CS(r_1+r_2+\ldots+r_k,k)$ as an induced subgraph. Also, from the partition of the matrix, it can contain $CS(7,3)$  or $CS(7,4)$ or $CS(n+2,n)$ as induced subgraph as well as the eccentricity matrices of them is principal submatrix of it. That's, there can be three cases.\\
 Case 1: $r_1=4$ and $k=3$.\\
 Then  from (11), the matrix $X$ admits an equitable partition into $3$ classes, hence the characteristic polynomial of $H$ is $x^{r_2+r_3+r_4}[x^3-3x^2-(4\sum_{1\leq i<j\leq3}^{3}r_ir_j+4\sum_{i=}^{3}r_i)x-(4\sum_{1\leq i<j\leq3}^{3}r_ir_j+16\prod_{i=1}^{3}r_i)]$. By the Descartes rule of sign, the maximum possible number of positive zeroes is 1, hence $H\in\mathcal{H}$.\\

 Case 2: $r_1=3$ and $k=4$.\\
 Similarly, in that case the characteristic polynomial of $H$ is $x^{r_2+r_3+r_4-1}[x^4-2x^3-(4\sum_{1\leq i<j\leq4}^{4}r_ir_j+3\sum_{i=}^{4}r_i)x^2-(16\sum_{1\leq i<j<l\leq4}^{4}r_ir_jr_l+4\sum_{1\leq i<j\leq4}^{4}r_ir_j)x-(4\sum_{1\leq i<j\leq4}^{4}r_ir_j+48\prod_{i=1}^{4}r_i)]$\\
 By the Descartes rule of sign, the maximum possible number of positive zeroes is 1, hence the graph $H$ has exactly one positive $\eps$-eigenvalue.\\
Case 3: $r_1=2$\\
The quotient matrix $Q$ of the matrix in (11) is
\begin{eqnarray}
Q=\left[
\begin{array}{ccccc}
1 &r_2&r_3&\ldots &r_k\\
2 &0&2r_3&\ldots &2r_k\\
2 &2r_2&0&\ddots &\vdots\\
\vdots &\vdots &\ddots &\vdots \vdots \\
2 &2r_1 &\ldots &\ldots \vdots & 0
\end{array}
\right]
\end{eqnarray}\\
 Similar to (i) and (ii), let $Q'=Q+D$ be matrix, where $D=diag(-1/2,2r_1,\ldots,2r_k)$. Then $Q'$ has rank 1. Hence the eigenvalues of $Q=Q'-D$ are smaller than those of $Q'$, which implies that $Q$ has at most (hence exactly) one positive eigenvalue. Therefore $H\in \mathcal{H}$.
 \item[\textit{iv.)}] Let $H$ be the mixed extension of $S_k$ type $(r_1,-r_2,r_3,\ldots, r_k)$. Then
,\begin{eqnarray}
\eps(H)=\left[
\begin{array}{ccc}
(J-I)_{r_1\times r_1} & J_{r_1\times r_2} &J_{r_1\times t}\\
J_{r_2\times r_1} & 2(J-I)_{r_2\times r_2} & 2J_{r_2\times t}\\
J_{t\times r_1} & 2J_{t\times r_2} & X_{t\times t}
\end{array}
\right]
\end{eqnarray}

Also, since this type contains some subgraphs of $(r_1,r_2,\ldots,r_k)$, hence $H\in\mathcal{H}$ if and only if
\item[$\circ$]$r_1=2$.
\item[$\circ$]$r_1=3$ and $k=5$.
\item[$\circ$]$r_1=4$ and $k=4$.\\
Let $r_1=2$. Then the matrix $Q'=Q+D$ has rank 1 such that $D=diag(-1/2,2,2r_3,\ldots,2r_k)$. Then the eigenvalues of $Q$ are smaller than those of $Q'$, hence $H\in \mathcal{H}$. Similarly, also for  $r_1=3$, $k=5$ and $r_1=4$, $k=4$, we get $H\in \mathcal{H}$.
\end{proof}

\begin{remark}
  $(-1,\underbrace{m,m,\ldots,m)}_{k-1}\cong W_{m+1}^{k}$. Also, from (9),
  \begin{equation*}
   spec_{\varepsilon}(W_{m+1}^{k})=\big\lbrace -2m^{(k-1)},0^{(k(m-1))},\alpha_1^{(1)},\alpha_2^{(1)}
   \big\rbrace
\end{equation*}
where, $\alpha_1=m(k-1)-\sqrt{m^2(k^2-2k+1)+km)}$ and $\alpha_2=m(k-1)+\sqrt{m^2(k^2-2k+1)+km)}$.
\end{remark}

\begin{corollary}
	Let $\mathcal{G^*}\subset \mathcal{G}$ be the family having at most one $\eps$-eigenvalue. $G$ belongs to $\mathcal{G^*}$ if and only if $G$ is one of the following.
	\item[i.] A mixed extension of $S_{k}$ of type $(r_1,r_2,r_3,\ldots, r_k)$ for $r_1\leq2$,$k\geq 2$; $r_1=3,k\leq 4$ or $r_1=4,k\leq 4$.
	\item[ii.]  A mixed extension of $S_{k}$ of type $(r_1,-r_2,r_3,\ldots, r_k)$ for $r_1\leq2$,$k\geq 2$; $r_1=3,k\leq 6$ or $r_1=4,k\leq 5$.
\end{corollary}

\begin{proof}
From Lemma 2.5, Lemma 2.6, Theorem 2.7 and Theorem 2.11 , we get the results.
\end{proof}
\section*{Conclusion}
This paper presents characterization of graphs having at most one (exactly) positive eccentricity eigenvalue w.r.t. the eccentricity matrix of the graphs . Firstly, we obtained that the graphs with diameter greater than equal $3$ and the self centered graphs cannot be have exactly one positive $\eps$-eigenvalue. The resultant graphs is the graphs with radius 1. Later, we restrict the graphs having the our properties. Finally, we complete the all graphs by help of  the concept of mixed extension. 

\section*{Acknowledgement}
We would like to thank to the Willem Haemers for sharing their pearls of wisdom with us during the course of this research and we are also immensely grateful for their comments on an earlier version of the manuscript.


\begin{thebibliography}{99}
\bibitem{1}
J.F. Wang, M. Lu, F. Belardo, M. Randić, The anti-adjacency matrix of a graph: eccentricity matrix, Discrete Appl. Math. 251 (2018) 299–309.

\bibitem{3}
J.F. Wang, L. Lu, M. Randić, G.Z. Li, Graph energy based on the eccentricity matrix, Discrete Math. 342 (2019) 2636–2646.

\bibitem{4}
I. Mahato, R. Gurusamy, M. R. Kannan, and S. Arockiaraj. Spectra of eccentricity matrices of graphs. preprint, Discrete Appl. Math. 285 (2020) 252-260.

\bibitem{5}
W.H. Haemers, Spectral characterization of mixed extensions of small graphs, Discrete Math.  342 (2019), 2760-2764.

\bibitem{6}
Roger A. Horn and Charles R. Johnson. Matrix analysis. Cambridge University Press, Cambridge, second edition, 2013.

\bibitem{7} W.H. Haemers, S. Sorgun, H. Topcu, On the spectral characterization of mixed extensions of $P_3$, Elect. J. Comb.,26 (3) (2019),P3.16.

\bibitem{11} H.Topcu, S. Sorgun, W.H. Haemers , The graphs cospectral with the pineapple graph, Discrete Appl. Math, 269 (2019) 52-59.

\bibitem{8} M. A. Henning, J. Southey, A characterization of the non-trivial diameter two graphs of minimum size, Discrete Appl. Math, 187 (2015) 91-95.

\bibitem{9} A. E. Brouwers , W.H.Haemers, Spectra of graphs, Springer,2012.

\bibitem{10} J. Wang, M. Lu, L. Lu ,F. Belardo, Spectral properties of the eccentricity matrix of graphs, Discrete Appl. Math. 279 (2020) 168-177.

\bibitem{zoran} Z. Stanic, Some Notes on Minimal Self Centered Graphs, AKCE J. Graphs. Combin.,1 (2010),97-102.

\end{thebibliography}
\end{document}